\documentclass[11pt]{amsart}
\usepackage[foot]{amsaddr}
\usepackage[english]{babel}

\usepackage[
backend=biber,
style=numeric,
giveninits=false,   
maxbibnames=99,     
minbibnames=99,
maxcitenames=99,
mincitenames=99,
doi=true,
url=false,
isbn=false
]{biblatex}

\DeclareFieldFormat[article]{title}{#1}

\DeclareFieldFormat{journaltitle}{\mkbibemph{#1}}

\DeclareFieldFormat[article]{volume}{\mkbibbold{#1}}

\DeclareFieldFormat{pages}{#1}

\DeclareFieldFormat{eid}{Article number\space #1}

\DeclareFieldFormat{doi}{\textsc{doi}:\space \href{https://doi.org/#1}{#1}}

\DeclareBibliographyDriver{article}{%
	\usebibmacro{bibindex}%
	\usebibmacro{begentry}%
	\printnames{author}%
	\setunit{\addperiod\space}%
	\printfield{title}%
	\setunit{\addperiod\space}%
	\printfield{journaltitle}%
	\setunit{\addspace}%
	\printfield{volume}%
	\setunit{\addspace}%
	\printtext[parens]{\printdate}%
	\setunit{\addspace}%
	\iffieldundef{pages}
	{\printfield{eid}}
	{\printfield{pages}}%
	\newunit\newblock
	\printfield{doi}%
	\addperiod%
	\usebibmacro{finentry}%
}
\allowdisplaybreaks

\usepackage[
pdftex,hyperfootnotes]{hyperref}
\hypersetup{
	colorlinks=true,
	linkcolor=NavyBlue, 
	urlcolor=RoyalPurple,
	citecolor=OliveGreen,
	pdftitle={Banded totally positive matrices and mixed multiple orthogonal polynomials},
	bookmarks=true,
}
\usepackage[usenames,dvipsnames,svgnames,table,x11names]{xcolor}
\usepackage{pgfplots}
\usepackage{tikz}
\usepackage{tikz-3dplot}
\usetikzlibrary{automata,quotes, chains,matrix,calc,shadows,shapes.callouts,shapes.geometric,shapes.misc,positioning,patterns,decorations.shapes,
	decorations.pathmorphing,decorations.markings,decorations.fractals,decorations.pathreplacing,shadings,fadings,arrows.meta,bending}

\usepackage{nicematrix,drawmatrix}
\usepackage[utf8]{inputenc}

\usepackage{comment}
\usepackage[textwidth=17.5cm,textheight=22.275cm,
height=
24.275cm,
width=18cm]{geometry}

\usepackage{amssymb,latexsym,amsmath,amsthm,bm}
\usepackage{mathrsfs}
\usepackage{mathtools,arydshln,mathdots}
\mathtoolsset{showonlyrefs}
\usepackage{tcolorbox}

\usepackage[table]{xcolor}

\usepackage{Baskervaldx}
\usepackage[]{newtxmath}

\usepackage{bigints}

\theoremstyle{plain}

\newtheorem{teo}{Theorem}[section]

\newtheorem{pro}[teo]{Proposition}
\newtheorem{defi}[teo]{Definition}
\newtheorem{rem}[teo]{Remark}

\renewcommand{\d}{\operatorname{d}}

\newcommand{\diag}{\operatorname{diag}}

\newcommand{\N}{\mathbb{N}}
\newcommand{\R}{\mathbb{R}}

\allowdisplaybreaks[1]

\makeatletter
\DeclareRobustCommand{\gaussk}{\DOTSB\gaussk@\slimits@}
\newcommand{\gaussk@}{\mathop{\vphantom{\sum}\mathpalette\bigcal@{K}}}

\newcommand{\bigcal@}[2]{%
	\vcenter{\m@th
		\sbox\z@{$#1\sum$}%
		\dimen@=\dimexpr\ht\z@+\dp\z@
		\hbox{\resizebox{!}{0.8\dimen@}{$\mathcal{K}$}}%
	}%
}
\newcommand{\cfracplus}{\mathbin{\cfracplus@}}
\newcommand{\cfracplus@}{%
	\sbox\z@{$\dfrac{1}{1}$}%
	\sbox\tw@{$+$}%
	\raisebox{\dimexpr\dp\tw@-\dp\z@\relax}{$+$}%
}
\newcommand{\cfracdots}{\mathord{\cfracdots@}}
\newcommand{\cfracdots@}{%
	\sbox\z@{$\dfrac{1}{1}$}%
	\sbox\tw@{$+$}%
	\raisebox{\dimexpr\dp\tw@-\dp\z@\relax}{$\cdots$}%
}
\makeatother

\makeatletter
\newcommand*{\relrelbarsep}{.386ex}
\newcommand*{\relrelbar}{%
	\mathrel{%
		\mathpalette\@relrelbar\relrelbarsep
	}%
}
\newcommand*{\@relrelbar}[2]{%
	\raise#2\hbox to 0pt{$\m@th#1\relbar$\hss}%
	\lower#2\hbox{$\m@th#1\relbar$}%
}
\providecommand*{\rightrightarrowsfill@}{%
	\arrowfill@\relrelbar\relrelbar\rightrightarrows
}
\providecommand*{\leftleftarrowsfill@}{%
	\arrowfill@\leftleftarrows\relrelbar\relrelbar
}
\providecommand*{\xrightrightarrows}[2][]{%
	\ext@arrow 0359\rightrightarrowsfill@{#1}{#2}%
}
\providecommand*{\xleftleftarrows}[2][]{%
	\ext@arrow 3095\leftleftarrowsfill@{#1}{#2}%
}
\makeatother

\catcode`,\active

\catcode`\,12

\usepackage{appendix}

\usepackage{xcolor} 

\usepackage[normalem]{ulem}
\usepackage{soul} 

\usepackage{comment} 

\usepackage{tikz}
\usetikzlibrary{shapes,arrows}
\usepackage{verbatim}

\tikzstyle{block} = [draw, rectangle, 
minimum height=3em, minimum width=2em]

\usepackage{mathrsfs} 

\begin{document}

\title[Markov chains and  stochastic bidiagonal factorization]{Spectral theory for Markov chains with  transition matrix admitting a stochastic bidiagonal factorization}

\author[A Branquinho]{Amílcar Branquinho$^{1}$}
\address{$^1$Departamento de Matemática,
 Universidade de Coimbra, 3001-454 Coimbra, Portugal}
\email{ajplb@mat.uc.pt}

\author[A Foulquié]{Ana Foulquié-Moreno$^{2}$}
\address{$^2$Departamento de Matemática, Universidade de Aveiro, 3810-193 Aveiro, Portugal}
\email{foulquie@ua.pt}

\author[M Mañas]{Manuel Mañas$^{3}$}
\address{$^3$Departamento de Física Teórica, Universidad Complutense de Madrid, Plaza Ciencias 1, 28040-Madrid, Spain \&
 Instituto de Ciencias Matematicas (ICMAT), Campus de Cantoblanco UAM, 28049-Madrid, Spain}
\email{manuel.manas@ucm.es}

\keywords{Markov chains, bidiagonal stochastic matrices, banded  matrices, oscillatory matrices, totally nonnegative matrices, multiple orthogonal polynomials, Favard spectral representation, positive bidiagonal factorization,  Karlin--McGregor representation formula, recurrent chains, ergodic chains}

\subjclass{42C05, 33C45, 33C47, 60J10, 60Gxx, 47B39, 47B36}

\enlargethispage{1.25cm}

\begin{abstract}
The recently established spectral Favard theorem for bounded banded matrices admitting a positive bidiagonal factorization is applied to a broader class of Markov chains with bounded banded transition matrices—extending beyond the classical birth-and-death setting—to those that allow a positive stochastic bidiagonal factorization.
In the finite case, the Karlin–McGregor spectral representation is derived. The recurrence of the Markov chain is established, and explicit formulas for the stationary distributions are provided in terms of orthogonal polynomials.
Analogous results are obtained for the countably infinite case. In this setting, the chain is not necessarily recurrent, and its behavior is characterized in terms of the associated spectral measure.
Finally, ergodicity is examined through the presence of a mass at $1$ in the spectral measure, corresponding to the eigenvalue $1$ with both right and left eigenvectors.
\end{abstract}
 
 \maketitle
 
 

\thispagestyle{empty}
 \section{Introduction}
 
  The spectral analysis of Markov chains and random walks has a long history,
 originating in the work of Karlin and McGregor and further developed in a
 systematic way using orthogonal polynomials and operator-theoretic methods;
 see \cite{KarlinMcGregor,DominguezIglesia2022}. In this paper we study discrete-time Markov chains on a finite or countable state space whose transition matrix is a bounded banded stochastic matrix admitting a \emph{stochastic bidiagonal factorization}. This setting extends the classical birth--and--death (tridiagonal) framework to processes with finitely many upward and downward jumps, while preserving a strong positivity structure. Our approach combines tools from total positivity and oscillation theory, where bidiagonal factorizations play a structural role \cite{Gantmacher,Fallat,Pinkus}, with the spectral viewpoint provided by mixed multiple orthogonal polynomials \cite{nikishin_sorokin,Ismail,andrei_walter,afm}. For oscillatory banded matrices, eigenvalues are simple and positive and eigenvectors exhibit controlled sign variation, making it possible to develop a spectral description that is well suited to probabilistic interpretations.
 
 Our starting point is the spectral Favard theorem established for bounded banded matrices admitting a positive bidiagonal factorization \cite{BFM1,BFM4}. In that framework one associates to the left and right recursion polynomials a matrix of positive measures for which these polynomials form mixed multiple orthogonal families, yielding a Karlin--McGregor type spectral representation. Here we adapt this machinery to the stochastic context by proving that any positive bidiagonal factorization of a banded stochastic matrix can be \emph{renormalized} into a product of bidiagonal \emph{stochastic} factors. This normalization procedure, implemented through suitable diagonal conjugations built recursively from the bidiagonal factors, provides a canonical stochastic factorization and makes explicit the link between oscillatory matrix structure and Markov evolution. The banded Hessenberg and tetradiagonal situations previously analyzed in \cite{BFM2,BFM3} serve as guiding examples. In addition, the present banded framework clarifies and streamlines some normalization and truncation aspects in the Markov-chain interpretation that in the Hessenberg setting require extra care: here we provide a systematic diagonal-conjugation procedure producing stochastic bidiagonal factors and a consistent treatment of finite truncations via Perron--Frobenius renormalization, which is essential for the subsequent spectral and probabilistic conclusions.
 
 A key step is an explicit normalization procedure that turns a given positive bidiagonal factorization into a stochastic one by conjugation with suitable diagonal matrices built recursively from the bidiagonal factors. This yields a stochastic factorization
 \[
 T=\hat{\mathscr L}_1\cdots\hat{\mathscr L}_p\hat{\mathscr U}_q\cdots \hat{\mathscr U}_1,
 \]
 where each factor is a positive bidiagonal stochastic matrix. Besides providing a structural description of the dynamics as a concatenation of elementary (pure birth / pure death) steps, this factorization is the bridge between probabilistic properties of the chain and the oscillation/spectral theory of the underlying banded operator.
 
 For finite truncations, although leading principal submatrices are only semi-stochastic, Perron--Frobenius theory allows one to renormalize them into genuinely stochastic matrices by a diagonal similarity transform involving the positive Perron eigenvector. In this finite setting we derive an explicit Karlin--McGregor spectral representation for the iterated transition probabilities and for the corresponding generating functions, written in terms of the mixed multiple orthogonal polynomials and the associated discrete matrix-valued spectral measures. As consequences, we establish irreducibility and aperiodicity, prove recurrence, and obtain closed formulas for the unique stationary distribution in terms of spectral data (left/right eigenvectors and Christoffel-type coefficients), together with geometric convergence rates governed by the second largest eigenvalue.
 
 In the countably infinite case we obtain analogous spectral representations with respect to the limiting matrix of measures supported in $[0,1]$. In contrast with the finite situation, the chain need not be recurrent, and we characterize recurrence versus transience through a sharp integrability condition involving the $(1,1)$-entry of the spectral measure, in the spirit of the classical random-walk theory of Karlin and McGregor \cite{KarlinMcGregor}. Finally, we analyze ergodicity (positive recurrence) via the presence of a mass at $1$ in the spectral measure, which corresponds to the eigenvalue $1$ carrying both right and left eigenvectors; in that case we identify the stationary distribution from the corresponding spectral weights.
 
 Summarizing, the main contributions of the paper are:
 \begin{itemize}
 	\item to connect banded Markov chains admitting stochastic bidiagonal factorizations with the spectral Favard theorem for positive bidiagonal factorizations \cite{BFM1,BFM4};
 	\item to derive Karlin--McGregor type formulas (finite and countable cases) in terms of mixed multiple orthogonal polynomials and matrix-valued spectral measures \cite{afm,KarlinMcGregor};
 	\item to establish recurrence and provide explicit stationary distributions in the finite case, and to characterize recurrence/ergodicity in the infinite case through the spectral measure (including the mass at $1$ criterion).
 \end{itemize}

\subsection{Elements of Markov chains}


A countable Markov chain is a discrete-time stochastic process represented by a sequence of random variables $\lbrace X_n\rbrace_{n\in\mathbb N_0}$, where the distribution of the next state depends only on the current state. See  \cite{gallager,feller}. In this paper we focus exclusively on finite Markov chains \cite{haggstrom,kirwood}
	where the random variables $\lbrace X_n\rbrace_{n\in\mathbb N_0}$ take values in the finite state space $\lbrace 0,1,\ldots,N\rbrace\subset\mathbb N_0$. Each element of this set is called a ``state.'' For a discussion of semi-infinite Markov chains, see \cite{JP,BFM5}.

The transition probabilities,
\[
P_{i,j}\coloneq \Pr(X_{n+1}=j\mid X_{n}=i),\qquad i,j\in\{0,1,\ldots,N\},
\]
form an $(N+1)\times (N+1)$ stochastic matrix $P\coloneq \left[ \begin{matrix} P_{i,j} \end{matrix} \right]_{i,j=0}^N$, which satisfies
\[
\begin{aligned}
	P_{i,j}&\geq0, &
	\sum_{k=0}^N P_{i,k}&=1, &
	i,j&\in\{0,1,\ldots,N\}.
\end{aligned}
\]

The first condition means that $P$ is entrywise non-negative, while the second implies that
\begin{align*}
	\begin{aligned}
		P e^{[N]}&=e^{[N]}, & e^{[N]}&\coloneq
		\begin{bNiceMatrix}
			1 &\Cdots &1
		\end{bNiceMatrix}^\top\in\R^{N+1}.
	\end{aligned}
\end{align*}

By the Chapman--Kolmogorov equation, the probabilities of moving from one state to another in $r$ steps are given by the entries of $P^r$, namely
\begin{align*}
	\Pr(X_{n+r}=j\mid X_{n}=i)&=(P^r)_{i,j}.
\end{align*}
Using this, we define the period of a state.

\begin{defi}[Period and aperiodic states]
	\label{defperiodo}
	Let $i\in\{0,1,\ldots,N\}$ be a state. The {period} $d(i)$ of $i$ is the greatest common divisor of all integers $r\in\mathbb N$ such that $(P^r)_{i,i}>0$. If $d(i)=1$, the state $i$ is said to be {aperiodic}.
\end{defi}

\begin{rem}
	\label{condicionsuficienteaperiodicidad}
	If $P_{i,i}>0$, then state $i$ is aperiodic, since the set of return times contains $r=1$, and hence its greatest common divisor is $1$.
\end{rem}

\begin{defi}[First-passage-time probability]
	The {first-passage-time probability} is defined by
	\begin{align*}
		f^r_{i,j}=\Pr(X_{r}=j,X_{r-1}\neq j,\ldots ,X_{1}\neq j\mid X_{0}=i),
	\end{align*}
	for $i,j\in\{0,1,\ldots,N\}$ and $r\in\N_0$. It represents the probability of reaching state $j$ for the first time after $r$ transitions when starting from $i$.
\end{defi}

The corresponding generating functions are
\begin{align*}
	\begin{aligned}
		P_{i,j}(s)&\coloneq \sum_{r=0}^\infty (P^r)_{i,j}s^r, &
		F_{i,j}(s)&\coloneq \sum_{r=1}^\infty f^r_{i,j}s^r,
	\end{aligned}
\end{align*}
and they satisfy
\begin{align*}
	F_{i,i}(s)=1-\dfrac{1}{P_{i,i}(s)}.
\end{align*}
Using this, we introduce the following notions.

\begin{defi}[Recurrence and transience]
	\label{defrecurrente}
	A state $i$ is called {recurrent} if the probability of ever returning to $i$ is $1$, or equivalently if
	\begin{align*}
		\lim_{s\rightarrow1^{-}}F_{i,i}(s)=1.
	\end{align*}
	Otherwise, the state $i$ is said to be {transient}. If all states are recurrent, the chain is called recurrent; otherwise, it is transient.
\end{defi}

\begin{rem}
	\label{condicionsuficienteparalarecurrencia}
	From the relation above, state $i$ is recurrent if and only if $\lim\limits_{s\rightarrow1^{-}}P_{i,i}(s)=+\infty$.
\end{rem}

\begin{defi}[Ergodicity]
	\label{defergodico}
	A state $i$ is called {ergodic} if it is both aperiodic and recurrent. If all states are ergodic, the chain is called an ergodic chain.
\end{defi}

\begin{defi}[Classes of states]
	We say that state $j$ is {\emph{accessible}} from state $i$ ($i\to j$) if there exists $n\in\N$ such that $(P^n)_{i,j}>0$. Two states {\emph{communicate}} if $i\to j$ and $j\to i$.
	Two states $i$ and $j$ are said to be in the same {class} if they communicate.
\end{defi}

Then, for two states $i$ and $j$ in the same class, starting from state $i$, state $j$ can be reached after finitely many transitions, and vice versa.

\begin{teo}[Class properties]
	The {states in the same class are either all recurrent or all transient and have the same period}.
\end{teo}

\begin{defi}[Irreducible Markov chain]
	{If there is only one class, the Markov chain is said to be irreducible}.
\end{defi}

In the examples we discuss, there will be only one class of states, hence the chains will be irreducible.

\begin{defi}[Expected return times]
	Let the first time to reach state $j$ starting from state $i$ be the random variable
	\begin{align*}
		T_{i,j}\coloneq \min\{n\geq 1:X_{n}=j\mid X_{0}=i\}.
	\end{align*}
	The mean first-passage time from $i$ to $j$ is
	\begin{align*}
		\bar t _{i,j}\coloneq \operatorname {E}(T_{i,j})=\sum_{n=1}^{\infty }n f^n_{i,j}.
	\end{align*}
	In particular, when $i=j$ we write $\bar t_i$ instead of $\bar t_{i,i}$ and call it the expected (mean) return time of state $i$.
\end{defi}

This expectation $\bar t_i$ represents the expected number of steps needed for the chain to return to the (recurrent) state $i$.

\begin{defi}[Probability vectors]
	A vector $\pi = \begin{bNiceMatrix}
		\pi_0&	\pi_1 & \Cdots & \pi_N
	\end{bNiceMatrix}$, where $\pi_i \geq 0$ for $i \in \{0,1,\dots,N\}$ and $\pi e^{[N]} = \sum_{i=0}^N \pi_i = 1$, is referred to as a probability vector. This vector encodes a probability distribution on the state space, where $\pi_i$ is the probability of being in state $i$.
\end{defi}

Assuming a probability vector ${\pi}(0)$ as the initial distribution, after one transition the new probability of being in state $k$ is
\[
\pi_k(1)=\sum_{i=0}^{N}\pi_i(0)P_{i,k}.
\]
Thus, the new probability vector is $\pi(1)=\pi(0)P$. Note that $\pi(1) e^{[N]}=\pi(0)P e^{[N]}=\pi(0) e^{[N]}=1$.
Therefore, after $n$ transitions the probability vector is ${\pi}(n)={\pi}(0 )P^n$, which follows from the Chapman--Kolmogorov equation
\[
(P^{n+m})_{i,j}=\sum_{k=0}^{N} (P^{n})_{i,k}(P^{m})_{k,j}.
\]

\begin{defi}
	A steady (or stationary) state is an invariant probability vector ${\pi}$ satisfying ${\pi}= \pi P$.
\end{defi}

\begin{rem}
	In terms of entries, $\pi_i=\sum_{j=0}^{N}\pi_j P_{j,i}$. These are the balance equations.
\end{rem}

\begin{teo}
	\begin{enumerate}
		\item For an irreducible Markov chain, if there exists a steady state, it is unique and the Markov chain is recurrent.
		\item If the Markov chain is recurrent, there exists a unique steady state. In this case, the steady-state entries are $\pi_i=\frac{1}{\bar t_i}$, for $ i\in\{0,1,\ldots,N\}$.
		\item For ergodic Markov chains, we have the limit property
		\begin{align}\label{eq:limit}
			\lim_{r\to\infty}(P^r)_{i,j}=\pi_j.
		\end{align}
	\end{enumerate}
\end{teo}

\begin{rem}
	Note that for any initial probability vector $\pi(0)$ for a Markov chain satisfying \eqref{eq:limit}, the long-term distribution is
	\begin{align*}
		\pi_j(\infty)=\lim_{r\to\infty} \sum_{i=0}^{N}\pi_i(0)(P^r)_{i,j}
		=\sum_{i=0}^{N}\pi_i(0)\pi_j=\pi_j.
	\end{align*}
	In other words, the steady state is an equilibrium distribution to which all initial distributions converge as time goes to infinity.
\end{rem}

\begin{rem}
	Simple random walks can be viewed as birth--death Markov chains, where the transition matrix is tridiagonal. Starting from a given state $i$, the only possible transitions are to remain in the same state or to move to the neighboring states $i+1$ or $i-1$.
\end{rem}

\begin{rem}[Time reversal]
	Assuming the steady state satisfies $\pi_i>0$ for $i\in\{0,1,\dots,N\}$, define a matrix $Q$ with entries $Q_{i,j}=\frac{\pi_j}{\pi_i}P_{j,i}$.
	
	The matrix $Q$ is stochastic and, using Bayes' formula with the initial distribution $\pi$, it can be written as $Q_{i,j}=\Pr\left(X_n=j\mid X_{n+1}=i\right)$. Hence, $Q$ is the transition matrix of the time-reversed chain. Moreover, if there exists a stochastic matrix $Q$ satisfying the detailed balance equations $\pi_iQ_{i,j}=\pi_j P_{j,i}$ for a probability vector $\pi$, then $\pi$ is a steady state. A Markov chain is called reversible when $Q=P$, that is, when $\pi_iP_{i,j}=\pi_j P_{j,i}$ holds. In other words, the Markov chain and its time-reversal have the same finite-dimensional distributions. See \cite{bremaud,haggstrom}.
\end{rem}

\begin{rem}
	The product $PQ$ of two stochastic matrices $P$ and $Q$ is again stochastic, and therefore it is the transition matrix of another Markov chain. Given an initial probability vector $\pi$, the vector $\pi PQ$ represents a two-step evolution: first $\pi\mapsto \pi P=\pi'$, and then $\pi'\mapsto \pi'Q$.
\end{rem}

\section{Markov chains with banded transition matrices}

We discuss applications of the spectral representation developed above for bounded banded semi-infinite matrices admitting a positive bidiagonal factorization to banded stochastic matrices $T$ of the form
\begin{align}\label{eq:monic_Hessenberg}
	T&=
	\left[	\begin{NiceMatrix}
		T_{0,0} &\Cdots &&T_{0,q}& 0 & \Cdots &&\\
		\Vdots& &&&\Ddots &\Ddots& &\\
		T_{p,0}&&&&&&&\\
		0&\Ddots[shorten-end=15pt]&&&&&&\\
		\Vdots[shorten-end=5pt]&\Ddots[shorten-end=5pt]&&&&&&\\
		&&&&&&&\\
		&&&&&&&\\[6pt]
		&&&&&& &
	\end{NiceMatrix}\right]
\end{align}
with $T_{n,m}>0$, $n,m\in\N_0$, and
\[
\begin{aligned}
	T_{n,n-p}+\cdots+T_{n,n+q}&=1, &n&\in\N_0,
\end{aligned}
\]
where, for $n\in\{0,\dots, p-1\}$, negative indices are understood to yield zero entries.

Let us consider the semi-infinite vector
\[
e\coloneq \begin{bNiceMatrix}
	1& 1 &\Cdots
\end{bNiceMatrix}^\top.
\]
A nonnegative matrix $T$ is said to be stochastic if $Te=e$, i.e., all its entries are nonnegative and the sum of the entries in each row is $1$.

Positive bidiagonal semi-infinite matrices are semi-infinite matrices of the form
\[ \begin{aligned}
	\mathscr	L&\coloneq \left[\begin{NiceMatrix}[columns-width=auto]
		\mathscr	L_{0,0} &0&\Cdots&\\
		\mathscr	L_{1,0}& \mathscr	L_{1,1}&\Ddots&\\
		0& \mathscr	L_{2,1}& \mathscr	L_{2,2}& \\
		\Vdots[shorten-start=5pt,shorten-end=7pt] &\Ddots& \Ddots& \Ddots\\&&&
	\end{NiceMatrix}\right], &
	\mathscr U& \coloneq
	\left[\begin{NiceMatrix}[columns-width = auto]
		\mathscr U_{0,0}& 	\mathscr U_{0,1}&0&\Cdots&\\
		0&	\mathscr U_{1,1}& 	\mathscr U_{1,2}&\Ddots&\\
		\Vdots[shorten-end=5pt]&\Ddots[shorten-end=17pt]&	\mathscr U_{2,2}&\Ddots&\\
		& &\Ddots[shorten-end=12pt]  &\Ddots[shorten-end=8pt]  &\\&&&&
	\end{NiceMatrix}\right],
\end{aligned}\]
i.e. with $\mathscr L_{i,i},\mathscr L_{i,i+1},\mathscr U_{i,i},\mathscr U_{i+1,i}>0$. These matrices can be written as
\[\begin{aligned}
	\mathscr L &=L \;\Delta, & L&= \left[\begin{NiceMatrix}[columns-width=auto]
		1 &0&\Cdots&\\
		\frac{	\mathscr	L_{1,0}}{	\mathscr	L_{0,0}}& 1 &\Ddots&\\
		0&\frac{	\mathscr	L_{2,1}}{	\mathscr	L_{1,1}}& 1& \\
		\Vdots[shorten-start=5pt,shorten-end=7pt] &\Ddots& \Ddots& \Ddots\\&&&
	\end{NiceMatrix}\right], & \Delta&=\diag (\mathscr L_{0,0},\mathscr L_{1,1},\dots),\\
	\mathscr U &= \nabla \; U, & U&=	\left[\begin{NiceMatrix}[columns-width = auto]
		1& \frac{\mathscr U_{0,1}}{\mathscr U_{0,0}}&0&\Cdots&\\
		0& 1& \frac{\mathscr U_{1,2}}{\mathscr U_{1,1}}&\Ddots&\\
		\Vdots[shorten-end=5pt]&\Ddots[shorten-end=17pt]&1&\Ddots&\\
		& &\Ddots[shorten-end=12pt]  &\Ddots[shorten-end=8pt]  &\\&&&&
	\end{NiceMatrix}\right],  & \nabla&=\diag(\mathscr U_{0,0},\mathscr U_{1,1},\dots)
\end{aligned}\]
in terms of normalized positive bidiagonal matrices and diagonal matrices.

\begin{pro}
	Any positive bidiagonal factorization of a banded semi-infinite matrix
	\[
	T=\mathscr L_1\cdots \mathscr L_p \mathscr U_q\cdots \mathscr U_1
	\]
	can be written in a normalized way as
	\[
	T=L_1   \cdots  L_p\Delta  U_q\cdots   U_1,
	\]
	in terms of normalized bidiagonal semi-infinite matrices and a positive diagonal semi-infinite matrix $\Delta$, and vice-versa.
\end{pro}

\begin{proof}
	A positive bidiagonal factorization can be written as
	\[
	\begin{aligned}
		T&=\mathscr L_1\cdots \mathscr L_p \mathscr U_q\cdots \mathscr U_1\\&=
		\tilde L_1 \Delta_1 \tilde L_2 \Delta_2\cdots \tilde L_p\Delta_p\nabla_q \tilde U_q\cdots \nabla_2\tilde U_2\nabla_1 \tilde U_1,
	\end{aligned}
	\]
	with $\tilde L_i$, $i\in\{1,\dots,p\}$ and $\tilde U_j$, $j\in\{1,\dots,q\}$ normalized positive bidiagonal matrices, $\Delta_i$, $i\in\{1,\dots,p\}$, and $\nabla_j$, $j\in\{1,\dots,q\}$, positive diagonal matrices.
	Now, we define the bidiagonal matrices
	\[
	\begin{aligned}
		L_i &= (\Delta_1\cdots \Delta_{i-1})\tilde L_i (\Delta_1\cdots \Delta_{i-1})^{-1}, &i&\in\{1,\dots,p\},\\
		U_j &= (\nabla_{j-1}\cdots \nabla_1)^{-1}\tilde U_j (\nabla_{j-1}\cdots \nabla_1)  &j&\in\{1,\dots,q\},
	\end{aligned}
	\]
	where $L_1=\tilde L_1$ and $U_1=\tilde U_1$. By construction these new matrices are normalized positive bidiagonal matrices. Therefore, we can write
	\[
	\begin{aligned}
		T&=\mathscr L_1\cdots \mathscr L_p \mathscr U_q\cdots \mathscr U_1\\&=
		L_1   \cdots  L_p\Delta  U_q\cdots   U_1, & \Delta= \Delta_1\cdots \Delta_p\nabla_q\cdots\nabla_1.
	\end{aligned}
	\]
	
	Reversing the previous arguments, we conclude that any factorization
	\[
	\Delta= \Delta_1\cdots \Delta_p\nabla_q\cdots\nabla_1
	\]
	leads to a (non-normalized) positive bidiagonal factorization of $T$.
\end{proof}

The truncations (leading principal submatrices)
\begin{align}
	T^{[N]}&=
	\begin{bNiceMatrix}[columns-width = 1.5cm]
		T_{0,0} &\Cdots &&T_{0,q}& 0 & \Cdots &&0\\
		\Vdots& &&&\Ddots &\Ddots& &\Vdots\\
		T_{p,0}&&&&&&&\\
		0&\Ddots&&&&&&0\\
		\Vdots&\Ddots&&&&&&T_{N-q,N}\\
		&&&&&&&\Vdots\\
		&&&&&&&\\[6pt]
		0&\Cdots&&&0&T_{N,N-p}&\Cdots &T_{N,N}
	\end{bNiceMatrix}
\end{align}
that take only the first $N+1$ rows and columns,
have a corresponding normalized bidiagonal factorization in terms of the truncations of the corresponding bidiagonal factors
\[
T^{[N]}=L_1^{[N]}\cdots  L_p^{[N]}\Delta^{[N]} U_q^{[N]}\cdots U_1^{[N]}.
\]
Hence, $T^{[N]}$ are oscillatory matrices and, consequently,
its eigenvalues $\lambda_k^{[N]}$, $k\in\{0,1,\dots,N\}$, are simple and positive. These eigenvalues are the zeros of the characteristic polynomial $P_{N+1}(x)$. We also order them as
\[
\lambda_0^{[N]} > \lambda_1^{[N]} > \dots > \lambda_{N}^{[N]}.
\]

\begin{defi}
	Given any semi-infinite vector
	\[
	v=\begin{bNiceMatrix}
		v_0\\v_1\\\Vdots[shorten-end=5pt]
	\end{bNiceMatrix}
	\]
	we define a corresponding diagonal semi-infinite matrix with diagonal entries precisely the vector entries:
	\[
	\delta( v)\coloneq\diag(v_0,v_1,\dots).
	\]
	
	Given a positive bidiagonal factorization $T=\mathscr L_1\dots \mathscr L_p\mathscr U_q\cdots \mathscr U_1$, let us introduce recursively the following positive diagonal matrices
	\[
	\begin{aligned}
		\delta_j &\coloneq \delta(\mathscr U_j \delta_{j-1}e), & j&\in\{1,\dots, q\},\\
		\delta_{q+i}&\coloneq \delta(\mathscr L_{p+1-i}\delta_{q+i-1}e), & i&\in\{1,\dots,p\},
	\end{aligned}
	\]
	with the understanding that $\delta_0=I$, being $I$ the semi-infinite identity matrix.
\end{defi}
These diagonal matrices are the key to construct a stochastic bidiagonal factorization from a given positive bidiagonal factorization.

\begin{pro}
	Given a stochastic banded matrix $T$ with a positive bidiagonal factorization $T=\mathscr L_1\cdots \mathscr L_p\mathscr U_q\cdots \mathscr U_1$ there exists a positive stochastic bidiagonal factorization
	\[
	T=	\hat{\mathscr L}_1\cdots\hat{\mathscr L}_p\hat{\mathscr U}_q\cdots \hat{\mathscr U}_1
	\]
	where $\hat{\mathscr L}_i$, $i\in\{1,\dots,p\}$, $\hat{\mathscr U}_j$, $j\in\{1,\dots,q\}$ are bidiagonal stochastic semi-infinite matrices given by
	\[
	\begin{aligned}
		\hat{\mathscr U}_j &= \delta_j^{-1} \mathscr U_j \delta_{j-1}, & j&\in\{1,\dots,q\},\\
		\hat{\mathscr L}_i&= \delta_{q+p-i+1}^{-1} \mathscr L_i \delta_{q+p-i},& i&\in\{1,\dots,p\}.
	\end{aligned}
	\]
\end{pro}
\begin{proof}
	First let us observe that we can write $\mathscr U_1=\delta_1 \hat {\mathscr U}_1$ where $\hat{ \mathscr U}_1=\delta_1^{-1}\mathscr U_1$ is a stochastic matrix, i.e. $\hat{\mathscr U}_1 e=e$. Then $\mathscr U_2\mathscr U_1 =\mathscr U_2\delta_1\hat{\mathscr U}_1=\delta_2 \hat{\mathscr U}_2\hat{\mathscr U}_1$ where $\hat{\mathscr U}_2=\delta_2^{-1}\mathscr U_2\delta_1$ is again a stochastic matrix, i.e. $\hat{\mathscr U}_2 e=e$. In the next step, we find $\mathscr U_3\mathscr U_2\mathscr U_1= \mathscr U_3 \delta _2 \hat{\mathscr U}_2\hat{\mathscr U}_1=\delta_3\hat{\mathscr U}_3\hat{\mathscr U}_2\hat{\mathscr U}_1$ where $\hat{\mathscr U}_3=\delta_3^{-1}\mathscr U_3 \delta_2$ is again, by construction, a stochastic matrix. Iterating this process we get to $\mathscr U_q\cdots \mathscr U_1 =\delta_q\hat{\mathscr U}_q\cdots \hat{\mathscr U}_1$ where all the matrices $\hat{\mathscr U}_j$, $j\in\{1,\dots,q\}$, are stochastic.
	
	Now, we deal with $\mathscr L_p\mathscr U_q\cdots \mathscr U_1=\mathscr L_p\delta_q\hat{\mathscr U}_q\cdots \hat{\mathscr U}_1= \delta_{q+1}\hat{\mathscr L}_p\hat{\mathscr U}_q\cdots \hat{\mathscr U}_1$, where by construction $\hat{\mathscr L}_p=\delta_{q+1}^{-1}\mathscr L_p\delta_q$ is again a stochastic matrix. Iterating again this process we conclude that
	\[
	T=\delta_{q+p}\hat{\mathscr L}_1\cdots\hat{\mathscr L}_p\hat{\mathscr U}_q\cdots \hat{\mathscr U}_1
	\]
	where each factor $\hat{\mathscr L}_i$, $i\in\{1,\dots, p\}$ and $\hat{\mathscr U}_j$, $j\in\{1,\dots,q\}$ is a positive bidiagonal stochastic matrix. Then, $Te=\delta_{q+p}e$, but by assumption $T$ is stochastic and therefore $Te=e$. Consequently, $\delta_{q+p}e=e$ and $\delta_{q+p}=I$.
\end{proof}

\begin{rem}
	This applies to semi-infinite or finite banded matrices.
\end{rem}

%

\subsection{Finite Markov chains}

If $T$ is stochastic, the truncations $T^{[N]}$ need not be stochastic: they are \emph{substochastic}, i.e. all entries are nonnegative and the sum of the entries in some rows may be strictly less than $1$. Nevertheless, recalling that the finite matrix $T^{[N]}$ is nonnegative and nonsingular, the Perron--Frobenius theorem ensures the existence of an eigenvalue, which is positive, simple and equal to the spectral radius, say $\lambda^{[N]}_0$, and a corresponding eigenvector $u_0^{\langle N\rangle}$ with all its entries
\[
u_{0,n}^{\langle N\rangle}=B_{n-1}^{(1)}\big(\lambda^{[N]}_0\big)\rho^{[N]}_{0,1} +\cdots +B_{n-1}^{(q)}\big(\lambda^{[N]}_0\big)\rho^{[N]}_{0,q}
\]
positive, i.e.,
\[
T^{[N]}u_0^{\langle N\rangle}=\lambda_0^{[N]}u_0^{\langle N\rangle}.
\]
Notice that $u_0^{\langle N\rangle}=\delta(u_0^{\langle N\rangle})e^{[N]}$.

Here we have used the multiple orthogonal polynomials $B_{n-1}^{(1)}(x)$ and the Christoffel numbers $\rho^{[N]}_{0,q}$ (which are positive), see Appendix \ref{appendix:MOP}.

\begin{pro}
	The matrix
	\[
	\hat T^{[N]}\coloneq \frac{1}{\lambda_0^{[N]}}\delta(u_0^{\langle N\rangle})^{-1} T^{[N]} \delta(u_0^{\langle N\rangle})
	\]
	is stochastic.
\end{pro}
\begin{proof}
	It follows from
	\[
	\begin{aligned}
		\hat T^{[N]} e^{[N]}&=\frac{1}{\lambda_0^{[N]}}\delta(u_0^{\langle N\rangle})^{-1} T^{[N]} \delta(u_0^{\langle N\rangle}) e^{[N]}
		=\frac{1}{\lambda_0^{[N]}}\delta(u_0^{\langle N\rangle})^{-1} T^{[N]} u_0^{\langle N\rangle}
		\\& =\frac{1}{\lambda_0^{[N]}}\delta(u_0^{\langle N\rangle})^{-1} \lambda^{[N]}_0u_0^{\langle N\rangle}
		=\delta(u_0^{\langle N\rangle})^{-1} u_0^{\langle N\rangle}=e^{[N]}.
	\end{aligned}
	\]
\end{proof}

This normalization can be viewed as a discrete-time Doob $h$-transform
(with $h=u_0^{\langle N\rangle}$), which turns the semi-stochastic truncation
into a genuine Markov transition matrix, see \cite{Doob1957,Doob1984,LevinPeresWilmer2017}.

\begin{pro}
	The finite Markov chain with transition matrix $\hat T^{[N]}$ has only one class (hence it is irreducible).
\end{pro}
\begin{proof}
	As the matrices have a PBF, all the entries within the band are positive, and successive powers enlarge the width of the band. Hence, there exists $n\in\N$ such that $((\hat T^{[N]})^{n})_{i,j}>0$ for any pair of states $i,j\in\{0,1,\dots,N\}$, and therefore all states communicate.
\end{proof}

\begin{pro}
	The finite Markov chain with transition matrix $\hat T^{[N]}$ is aperiodic.
\end{pro}
\begin{proof}
	Having the transition matrix $\hat T^{[N]}$ a PBF, it has all its diagonal entries positive, $(\hat T^{[N]})_{i,i}>0$. Therefore, the set of return times contains $1$, and the period of each state is $1$.
\end{proof}

\begin{pro}
	The stochastic matrix $\hat T^{[N]}$ has a stochastic bidiagonal factorization.
\end{pro}
\begin{proof}
	Since $\hat T^{[N]}$ is stochastic and admits a positive bidiagonal factorization, the construction above yields a stochastic bidiagonal factorization.
\end{proof}

In terms of eigenvalues $\lambda^{[N]}_j$, and left and right eigenvector entries $u_{j,n}^{\langle N\rangle},w_{j,m}^{\langle N\rangle}$ of $T^{[N]}$, see \eqref{eq:eigenvectors}, and the spectral measure $\d\psi^{[N]}_{b,a}$, see \eqref{eq:discrete_mesures}, an extension of Karlin and McGregor for simple Markov chains is:
\begin{pro}
	The spectral representation of these finite stochastic matrices is
	\[
	\begin{aligned}
		((\hat T^{[N]})^k)_{n,m}&=
		\frac{u_{0,m}^{\langle N\rangle}}{u_{0,n}^{\langle N\rangle}}
		\sum_{a=1}^p\sum_{b=1}^{q}	\int B^{(b)}_n(x)x_N^k\d\psi^{[N]}_{b,a}(x)A^{(a)}_{m}(x)\\&=
		\frac{u_{0,m}^{\langle N\rangle}}{u_{0,n}^{\langle N\rangle}}	\sum_{j=0}^N
		u_{j,n}^{\langle N\rangle}\left(\frac{\lambda^{[N]}_j}{\lambda_0^{[N]}}\right)^kw_{j,m}^{\langle N\rangle}
		\\&=u_{0,m}^{\langle N\rangle}w_{0,m}^{\langle N\rangle}+
		\frac{u_{0,m}^{\langle N\rangle}}{u_{0,n}^{\langle N\rangle}}	\sum_{j=1}^N u_{j,n}^{\langle N\rangle}\left(\frac{\lambda^{[N]}_j}{\lambda_0^{[N]}}\right)^kw_{j,m}^{\langle N\rangle}
	\end{aligned}
	\]
	for $n,m\in\{0,\dots,N\}$.
\end{pro}

We introduce a new variable
\[
x_N\coloneq \frac{x}{\lambda_0^{[N]}}.
\]
Then, in terms of the recurrence polynomials $A^{(a)}_n(x),B^{(b)}_n(x)$ of the stochastic banded matrix $T$ and the spectral objects related to $T$, see Appendix \ref{appendix:MOP}, we find:

\begin{pro}
	For $|s|<1$, the corresponding transition probability generating functions are
	\[
	\begin{aligned}
		P_{m,n}(s)=	\frac{u_{0,n}^{\langle N\rangle}}{u_{0,m}^{\langle N\rangle}}\sum_{a=1}^p\sum_{b=1}^{q}	\int B^{(b)}_n(x)\frac{\d\psi^{[N]}_{b,a}}{1-sx_N}(x)A^{(a)}_{m}(x)
	\end{aligned}
	\]
	for $n,m\in\{0,\dots,N\}$.
	For $m\neq n$, the first passage generating functions are
	\begin{align*}
		F_{m,n}^{[N]}(s)&=\frac{u_{0,n}^{\langle N\rangle}}{u_{0,m}^{\langle N\rangle}}
		\dfrac{\sum_{a=1}^p\sum_{b=1}^{q}	\int B^{(b)}_n(x)\frac{\d\psi^{[N]}_{b,a}}{1-sx_N}(x)A^{(a)}_{m}(x)}{\sum_{a=1}^p\sum_{b=1}^{q}	\int B^{(b)}_n(x)\frac{\d\psi^{[N]}_{b,a}}{1-sx_N}(x)A^{(a)}_{n}(x)}
	\end{align*}
	for $n,m\in\{0,\dots,N\}$.
	For the case $n=m$, the first passage generating functions are
	\begin{align*}
		F_{m}^{[N]}(s)=1-\dfrac{1}{\sum_{a=1}^p\sum_{b=1}^{q}	\int B^{(b)}_m(x)\frac{\d\psi^{[N]}_{b,a}}{1-sx_N}(x)A^{(a)}_{m}(x)}.
	\end{align*}
	for $m\in\{0,\dots,N\}$.
	
\end{pro}
\begin{proof}
	It follows from the spectral representation and the sum of the geometric series.
\end{proof}

\begin{teo}[Recurrent chains]
	The finite Markov chain with transition matrix $\hat T^{[N]}$ is recurrent.
\end{teo}

\begin{proof}
	Observe that
	\[
	\sum_{a=1}^p\sum_{b=1}^{q}	\int B^{(b)}_m(x)\frac{\d\psi^{[N]}_{b,a}}{1-sx_N}(x)A^{(a)}_{m}(x)
	=
	\frac{1}{1-s}u^{{\langle N\rangle}}_{0,m}w^{{\langle N\rangle}}_{0,m}
	+\sum_{j=1}^{N}\frac{1}{1-s\frac{\lambda^{[N]}_j}{\lambda^{[N]}_0}}	u^{{\langle N\rangle}}_{j,m}w^{{\langle N\rangle}}_{j,m}
	\]
	in terms of the entries of the right and left eigenvectors, $u^{\langle N\rangle}_{k,m},w^{\langle N\rangle}_{k,m}$, respectively. As all the entries of the Perron left and right eigenvectors, $u^{\langle N\rangle}_{0,m},w^{\langle N\rangle}_{0,m}$, are positive we conclude that
	\[
	\lim\limits_{s\to 1^-}F_{mm}^{[N]}(s)=1,
	\]
	and all the states are recurrent.
\end{proof}

\begin{teo}[Ergodic chain]
	The finite Markov chain with transition matrix $\hat T^{[N]}$ is ergodic.
\end{teo}
\begin{proof}
	Ergodic chains are irreducible recurrent chains, and both properties have been proven.
\end{proof}

\begin{teo}[Stationary state]\label{teo:stationary_states_finite}
	The probability distribution \[ \pi^{[N]}=\begin{bNiceMatrix}
		\pi_{0}^{[N]}&\pi_{1}^{[N]}&\Cdots&\pi^{[N]}_{N}
	\end{bNiceMatrix}\coloneq
	w^{\langle N\rangle}_0\delta\left(u^{\langle N\rangle}_0\right), \]
	that has entries
	\begin{align*}
		\pi^{[N]}_m=u^{\langle N\rangle}_{0,m}w^{{\langle N\rangle}}_{0,m}=\sum_{a=1}^p\sum_{b=1}^q
		B_{m}^{(b)}\big(\lambda^{[N]}_0\big) A_{m}^{(a)}\big(\lambda^{[N]}_0\big)\rho^{[N]}_{0,b} \mu^{[N]}_{0,a}=
		\frac{ \alpha_N Q_{m,N}\big(\lambda^{[N]}_0\big)R_{m,N}\big(\lambda^{[N]}_0\big)
		}{
			P_{N}\big(\lambda^{[N]}_0\big)P'_{N+1}\big(\lambda^{[N]}_0\big)}.
	\end{align*}
	is stationary.
\end{teo}

\begin{proof}
	The eigenvalue property $\pi^{[N]}\hat T =\pi^{[N]}$ follows from
	\[
	\begin{aligned}
		\pi^{[N]}\hat T &= w^{\langle N\rangle}_0\delta(u^{\langle N\rangle}_0)\frac{1}{\lambda_0^{[N]}}\delta(u_0^{\langle N\rangle})^{-1} T^{[N]} \delta(u_0^{\langle N\rangle})\\&=
		\frac{1}{\lambda_0^{[N]}}	w^{\langle N\rangle}_0T^{[N]} \delta(u_0^{\langle N\rangle})=w^{\langle N\rangle}_0  \delta(u_0^{\langle N\rangle})=\pi^{[N]}.
	\end{aligned}
	\]
	Also it follows that $\pi^{[N]}_n>0$, $n\in\{0,1,\dots,N\}$ and that $\displaystyle \pi^{[N]}e^{[N]} =\sum_{n=0}^{N}\pi^{[N]}_n=\sum_{n=0}^{N} u_{0,n}w_{0,n}=1$.
\end{proof}

\begin{pro}
	The steady state $\pi_m$ is unique, and the expected return times are given by
	\begin{align*}
		(\bar t_m)_j &=\frac{1}{\pi^{[N]}_m}=\frac{1}{	u^{\langle N\rangle}_{0,m}w^{{\langle N\rangle}}_{0,m}}=\frac{1}{\sum_{a=1}^p\sum_{b=1}^q
			B_{m}^{(b)}\big(\lambda^{[N]}_0\big) A_{m}^{(a)}\big(\lambda^{[N]}_0\big)\rho^{[N]}_{0,b} \mu^{[N]}_{0,a}}= 	\frac{
			P_{N}\big(\lambda^{[N]}_0\big)P'_{N+1}\big(\lambda^{[N]}_0\big)}{ \alpha_N Q_{m,N}\big(\lambda^{[N]}_0\big)R_{m,N}\big(\lambda^{[N]}_0\big)
		}.
	\end{align*}
	
\end{pro}

\begin{proof}
	As was proved, the Markov chain is recurrent, so the steady state is unique, and its entries are reciprocal to these expected return times.
\end{proof}

\begin{pro}
	The stationary state is the limit of the iterated probabilities
	\[
	\pi^{[N]}_m=\lim_{k\to\infty} ((\hat T^{[N]})^k)_{n,m},
	\]
	and the convergence is geometric in terms of the second largest eigenvalue:
	\[
	\begin{aligned}
		\lim_{k\to\infty}\left(((\hat T^{[N]})^k)_{n,m}-\pi^{[N]}_m\right)&\xrightarrow[k\to\infty]{}
		\left(	\frac{\lambda^{[N]}_1}{\lambda^{[N]}_0}\right)^k	u^{{\langle N\rangle}}_{1,m}w^{{\langle N\rangle}}_{1,m}=	\left(	\frac{\lambda^{[N]}_1}{\lambda^{[N]}_0}\right)^k 	\frac{ \alpha_N Q_{m,N}\big(\lambda^{[N]}_1\big)R_{m,N}\big(\lambda^{[N]}_1\big)
		}{
			P_{N}\big(\lambda^{[N]}_1\big)P'_{N+1}\big(\lambda^{[N]}_1\big)}.
	\end{aligned}
	\]
\end{pro}
\begin{proof}
	Note that
	\[
	\begin{aligned}
		((\hat T^{[N]})^k)_{n,m}-\pi^{[N]}_m&=
		\sum_{j=1}^{N}
		\left(	\frac{\lambda^{[N]}_j}{\lambda^{[N]}_0}\right)^k	u^{{\langle N\rangle}}_{j,m}w^{{\langle N\rangle}}_{j,m}.
	\end{aligned}
	\]
	Then, as eigenvalues are simple and ordered, the second dominant term is related to the second largest eigenvalue.
\end{proof}

\begin{rem}
	Then, this steady state is the equilibrium state to which the evolution of any probability vector tends as time goes to infinity. This follows also from ergodicity.
\end{rem}

\begin{pro}
	The transition matrix of the time reversal Markov chain of the Markov chain with transition matrix $\hat T^{[N]}$ is
	\begin{align*}
		\tilde T^{[N]}= \frac{1}{\lambda^{[N]}_0}\delta(w_0^{\langle N\rangle})^{-1}(T^{[N]})^\top \delta(w_0^{\langle N\rangle})
	\end{align*}
	with entries
	\[
	(\tilde T^{[N]})_{i,j}= \frac{1}{\lambda^{[N]}_0}
	\frac{w^{{\langle N\rangle}}_{0,j}}{w^{{\langle N\rangle}}_{0,i}}(T^{[N]})_{j,i}.
	\]
\end{pro}

\begin{proof}
	As the steady state satisfies $\pi^{[N]}_i>0$ for $i\in\{0,1,\dots,N\}$, the time reversal Markov chain has transition matrix entries
	\begin{align*}
		(\tilde T^{[N]})_{i,j}&=\frac{\pi^{[N]}_j}{\pi^{[N]}_i}(\hat T^{[N]})_{j,i}=\frac{u^{\langle N\rangle}_{0,j}w^{{\langle N\rangle}}_{0,j}}{u^{\langle N\rangle}_{0,i}w^{{\langle N\rangle}}_{0,i}}(\hat T^{[N]})_{j,i}= \frac{1}{\lambda^{[N]}_0}
		\frac{u^{\langle N\rangle}_{0,j}w^{{\langle N\rangle}}_{0,j}}{u^{\langle N\rangle}_{0,i}w^{{\langle N\rangle}}_{0,i}}\frac{u^{\langle N\rangle}_{0,i}}{u^{\langle N\rangle}_{0,j}}(T^{[N]})_{j,i}\\
		&= \frac{1}{\lambda^{[N]}_0}
		\frac{w^{{\langle N\rangle}}_{0,j}}{w^{{\langle N\rangle}}_{0,i}}(T^{[N]})_{j,i}.
	\end{align*}
	Note that the matrix $\tilde T^{[N]}$ is stochastic, i.e.
	\[
	\begin{aligned}
		\tilde T^{[N]}e^{[N]}&= \frac{1}{\lambda^{[N]}_0}\delta(w_0^{\langle N\rangle})^{-1}(T^{[N]})^\top \delta(w_0^{\langle N\rangle})e^{[N]}
		=\delta(w_0^{\langle N\rangle})^{-1}(w_0^{\langle N\rangle})^\top=e^{[N]}.
	\end{aligned}
	\]
\end{proof}

\begin{pro}[Detailed balance]
	The following detailed balance equation is satisfied:
	\begin{align}\label{eq:detalied_balance}
		\pi^{[N]}_i (\hat T^{[N]})_{i,j} &= \pi^{[N]}_j (\tilde T^{[N]})_{j,i},
	\end{align}
	for $i,j\in\{0,1,\dots,N\}$.
\end{pro}

\subsection{Countable infinite Markov chains}

We now turn to the Markov chain that has a bounded banded transition matrix with a stochastic bidiagonal factorization (or, equivalently, a stochastic matrix with a positive bidiagonal factorization).

Being $T$ stochastic, we have $\|T\|_\infty =1$. Since its truncations $T^{[N]}$ are substochastic, their eigenvalues satisfy $\{\lambda_k^{[N]}\}_{k=0}^N\subset (0,1)$. Recall that the eigenvalues $\{\lambda_k^{[N]}\}_{k=0}^N$ are interlaced by $\{\lambda_k^{[N+1]}\}_{k=0}^{N+1}$. Hence, the support of the limit measure belongs to $[0,1]$.
From the inductive limit $T^{[N]}\to T$ we conjecture that $\lambda_0^{[N]}\xrightarrow[N\to\infty]{}1$.

\begin{teo}[Karlin--McGregor spectral representation]
	We have the following spectral representations for several probability quantities of the Markov chain:
	\begin{enumerate}
		\item 	The spectral representation for the iterated probabilities is given by
		\[
		(T^k)_{n,m}=
		\sum_{a=1}^p\sum_{b=1}^{q}	\int_0^1B^{(b)}_n(x)x^k\d\psi_{b,a}(x)A^{(a)}_{m}(x)
		\]
		for $n,m\in\N_0$.
		\item The corresponding generating functions are
		\[
		(T(s))_{n,m}=
		\sum_{a=1}^p\sum_{b=1}^{q}	\int_0^1B^{(b)}_n(x)\frac{\d\psi_{b,a}(x)}{1-sx}A^{(a)}_{m}(x)
		\]
		\item The first passage generating functions, for $n\neq m$, are
		\[
		(F(s))_{n,m}=
		\frac{	\sum_{a=1}^p\sum_{b=1}^{q}	\int_0^1B^{(b)}_n(x)\frac{\d\psi_{b,a}(x)}{1-sx}A^{(a)}_{m}(x)}{
			\sum_{a=1}^p\sum_{b=1}^{q}	\int_0^1B^{(b)}_m(x)\frac{\d\psi_{b,a}(x)}{1-sx}A^{(a)}_{m}(x)}
		\]
		\item The first passage generating functions, for $n= m$, are
		\[
		(F(s))_{m}=1-\frac{1}{\sum_{a=1}^p\sum_{b=1}^{q}	\int_0^1B^{(b)}_m(x)\frac{\d\psi_{b,a}(x)}{1-sx}A^{(a)}_{m}(x)}.
		\]
	\end{enumerate}
\end{teo}
\begin{proof}
	It is a direct consequence of the Favard spectral theorem described in the introduction, see \cite{BFM1}. That is, if $T$ has a PBF and an appropriate choice of the initial conditions for the recurrence polynomials is made, then there exists a matrix of measures such that the recurrence polynomials are the mixed-type multiple orthogonal polynomials with respect to this matrix of measures.
\end{proof}

As in the finite case we can prove that
\begin{pro}
	The Markov chain is irreducible, i.e. all states communicate.
\end{pro}

\begin{teo}
	The Markov chain is recurrent if and only if the integral
	\begin{align*}
		\int_0^{1}\frac{\d\psi_{1,1}(x)}{1-x}
	\end{align*}
	diverges. Otherwise it is transient.
\end{teo}
\begin{proof}
	As the Markov chain is irreducible we only need to prove this fact for the state $0$, and then it propagates to all the other states. Thus we need to check whether $\lim_{s\to1^-} (F(s))_0=1$ or not. However, for the first passage generating function for the $0$ state
	\[
	(F(s))_{0}=1-\frac{1}{\sum_{a=1}^p\sum_{b=1}^{q}	\int_0^1B^{(b)}_0(x)\frac{\d\psi_{b,a}(x)}{1-sx}A^{(a)}_{0}(x)},
	\]
	recalling the initial conditions for the mixed-type multiple orthogonal polynomials, we find
	\[
	\sum_{a=1}^p\sum_{b=1}^{q}\int_0^1B^{(b)}_0(x)\frac{\d\psi_{b,a}(x)}{1-sx}A^{(a)}_{0}(x)=
	\int_0^1\frac{\d\psi_{1,1}(x)}{1-sx},
	\]
	and
	\[
	(F(s))_{0}=1-\frac{1}{\int_0^1\frac{\d\psi_{1,1}(x)}{1-sx}}.
	\]
	The result follows immediately.
\end{proof}

Regarding stationary states we find:
\begin{teo}[Ergodic Markov chains]
	The Markov chain described is ergodic (or positive recurrent) if and only if $1$ is a mass point of $\d\psi_{b,a}$, with masses $m_{1,1}>0$ and $m_{a,b}\geqslant 0$, for $a\in\{1,\dots,p\}$, $b\in\{1,\dots,q\}$. In that case, the corresponding stationary distribution is
	\[	\begin{aligned}
		\pi&=\left[\begin{NiceMatrix}
			\pi_1 &\pi_2 &\Cdots
		\end{NiceMatrix}\right], &
		\pi_{n}&=\sum_{a=1}^p\sum_{b=1}^q B_n^{(b)}(1)m_{b,a}A_n^{(a)}(1).
	\end{aligned}\]
\end{teo}
\begin{proof}
	
	Following \cite{KarlinMcGregor}, see also \cite{gallager}, for a recurrent chain the expected first passage times are all finite if
	\[
	0<\lim_{n\to\infty} T_{00}^{2n}<+\infty,
	\]
	that is
	\begin{align*}
		0<	\lim_{n\to\infty} \int_0^1\d\psi_{1,1} (x)x^{2n}<+\infty.
	\end{align*}
	Since $x^{2n}\xrightarrow[n\to\infty]{} 0$ monotonically for $0<x<1$, the chain is ergodic if and only if $\psi_{1,1}$ has a jump at $x=1$.
	Therefore, from Theorem \ref{teo:stationary_states_finite} we see that
	\begin{align*}
		\lim_{N\to\infty} \pi^{[N]}_{n}=		\lim_{N\to\infty}\sum_{a=1}^p\sum_{b=1}^q
		B_{m}^{(b)}\big(\lambda^{[N]}_0\big) A_{m}^{(a)}\big(\lambda^{[N]}_0\big)\rho^{[N]}_{0,b} \mu^{[N]}_{0,a}.
	\end{align*}
	Due to the fact that $1$ is a mass point, $\lim_{N\to\infty}\lambda^{[N]}_0=1$
	and $\lim_{N\to\infty} \rho^{[N]}_{0,b}\mu^{[N]}_{0,a}=m_{b,a}$.
\end{proof}

\appendix

\section{Totally nonnegative and oscillatory matrices}

Totally nonnegative (TN) matrices are those whose minors are all nonnegative. We denote by $\operatorname{InTN}$ the set of nonsingular TN matrices. Oscillatory matrices are those that are totally nonnegative, irreducible, and nonsingular; we denote this class by $\operatorname{IITN}$ (irreducible invertible totally nonnegative). Equivalently, a matrix $T$ is oscillatory if it is totally nonnegative and there exists $n\in\N$ such that $T^n$ is totally positive (i.e., all minors are positive). By the Cauchy--Binet theorem one deduces that these classes are invariant under matrix multiplication. In particular, the product of matrices in $\operatorname{InTN}$ belongs again to $\operatorname{InTN}$ (and analogous statements hold for TN and oscillatory matrices).

The following results are crucial:
\begin{itemize}
	\item \textbf{Gantmacher--Krein criterion.}
	A totally nonnegative matrix is oscillatory if and only if it is nonsingular and its entries on the first subdiagonal and the first superdiagonal are positive.
	
	\item \textbf{Gauss--Borel factorization.}
	The matrix $T$ belongs to $\operatorname{InTN}$ if and only if it admits a Gauss--Borel factorization $T=L^{-1}U^{-1}$ with $L,U\in\operatorname{InTN}$, where $L$ and $U$ are lower and upper triangular, respectively.
	
	\item \textbf{Spectrum.}
	If $T\in\R^{N\times N}$ is oscillatory, then $T$ has $N$ distinct positive eigenvalues.
	
	\item \textbf{Interlacing.}
	If $T\in\R^{N\times N}$ is oscillatory, then its eigenvalues strictly interlace those of the two principal submatrices of order $N-1$, namely $T(1)$ and $T(N)$, obtained from $T$ by deleting the first row and column, or the last row and column, respectively.
\end{itemize}

We now introduce some notation. For a totally nonzero vector $u\in\R^n$ (i.e., $u_i\neq 0$ for all $i$), its total sign variation is
\[
v(u)=\#\{\,i\in\{1,\dots,n-1\}: u_i u_{i+1}<0\,\}.
\]
For a general vector $u\in\R^n$, we define $v_m(u)$ (resp.\ $v_M(u)$) as the minimum (resp.\ maximum) of $v(y)$ over all totally nonzero vectors $y$ that coincide with $u$ on its nonzero entries. If $v_m(u)=v_M(u)$, we write $v(u)\coloneq v_m(u)=v_M(u)$.

Finally, regarding eigenvectors: let $T\in\R^{N\times N}$ be oscillatory, and let $u^{(k)}$ (resp.\ $w^{(k)}$) be a right (resp.\ left) eigenvector corresponding to $\lambda_k$, the $k$-th largest eigenvalue of $T$. Then:
\begin{itemize}
	\item We have
	\[
	v_m\!\big(u^{(k)}\big)=v_M\!\big(u^{(k)}\big)=v\!\big(u^{(k)}\big)=k-1,
	\qquad
	v_m\!\big(w^{(k)}\big)=v_M\!\big(w^{(k)}\big)=v\!\big(w^{(k)}\big)=k-1.
	\]
	Moreover, the first and last entries of $u^{(k)}$ (and of $w^{(k)}$) are nonzero, and $u^{(1)}$ and $u^{(N)}$ (as well as $w^{(1)}$ and $w^{(N)}$) are totally nonzero; the remaining eigenvectors may have a zero entry.
	
	\item By the Perron--Frobenius theorem, $u^{(1)}$ (and $w^{(1)}$) can be chosen entrywise positive. The remaining eigenvectors $u^{(k)}$ (and $w^{(k)}$), $k=2,\dots,N$, have at least one sign change; in fact, $u^{(N)}$ (and $w^{(N)}$) has strictly alternating signs.
\end{itemize}

\section{Recursion relations and their consequences}\label{appendix:MOP}

Let us introduce the recursion polynomials associated with the banded matrix 
\begin{align}\label{eq:monic_Hessenberg}
	T&=
	\left[	\begin{NiceMatrix}
		T_{0,0} &\Cdots &&T_{0,q}& 0 & \Cdots &&\\
		\Vdots& &&&\Ddots &\Ddots& &\\
		T_{p,0}&&&&&&&\\
		0&\Ddots[shorten-end=15pt]&&&&&&\\
		\Vdots[shorten-end=5pt]&\Ddots[shorten-end=5pt]&&&&&&\\
		&&&&&&&\\
		&&&&&&&\\[6pt]
		&&&&&& &
	\end{NiceMatrix}\right]
\end{align}
as the entries of semi-infinite left and right eigenvectors:
\[ \begin{aligned}
	A^{(a)}&=\left[\begin{NiceMatrix}
		A^{(a)}_0 & 	A^{(a)}_1& \Cdots
	\end{NiceMatrix}\right], & a&\in\{1,\dots, p\},&
	B^{(b)}&=\begin{bNiceMatrix}
		B^{(b)}_0 \\[2pt]	B^{(b)}_1\\ \Vdots
	\end{bNiceMatrix}, & b&\in\{1,\dots, q\},
\end{aligned}\]
which are left and right eigenvectors of \( T \) with eigenvalue \( x \), i.e.
\[ \begin{aligned}
	A^{(a)}T&=xA^{(a)},& a&\in\{1,\dots, p\},&
	TB^{(b)}&=xB^{(b)}, &b&\in\{1,\dots, q\}.
\end{aligned}\]

We assume that the semi-infinite matrix $T$ admits a normalized positive bidiagonal factorization (PBF) 
\begin{align}\label{eq:bidiagonal}
	T= L_{1} \cdots L_{p} \Delta U_q\cdots U_1,
\end{align}
with $\Delta=\diag(\Delta_0,\Delta_1,\dots)$ and normalized bidiagonal semi-infinite matrices (with diagonal entries equal to one) given by 
\begin{equation}\label{eq:bidiagonal_factors}
	\begin{aligned}
		L_k&\coloneq \left[\begin{NiceMatrix}[columns-width=auto]
			1 &0&\Cdots&\\
			L_{k|1,0} & 1 &\Ddots&\\
			0& L_{k|2,1}& 1& \\
			\Vdots[shorten-start=5pt,shorten-end=7pt] &\Ddots& \Ddots& \Ddots\\&&&
		\end{NiceMatrix}\right], & 
		U_k& \coloneq
		\left[\begin{NiceMatrix}[columns-width = auto]
			1& U_{k|0,1}&0&\Cdots&\\
			0& 1& U_{k|1,2}&\Ddots&\\
			\Vdots[shorten-end=5pt]&\Ddots[shorten-end=17pt]&1&\Ddots&\\
			& &\Ddots[shorten-end=12pt]  &\Ddots[shorten-end=8pt]  &\\&&&&
		\end{NiceMatrix}\right], 
	\end{aligned}
\end{equation}
and such that the positivity constraints $L_{k|i,j},U_{k|i,j},\Delta_i>0$, for $i\in\N_0$, are satisfied. 

The entries of these left and right eigenvectors are polynomials in the eigenvalue \(x\), known as left and right recursion polynomials, respectively, and are determined by the initial conditions
\begin{align}
	\label{eq:initcondtypeI}
	\begin{aligned}
		\begin{cases}
			A^{(1)}_0=1 , \\
			A^{(1)}_1= \nu^{(1)}_1 , \\
			\hspace{.895cm} \vdots \\
			A^{(1)}_{p-1}=\nu^{(1)}_{p-1} ,
		\end{cases}
		&&
		\begin{cases}
			A^{(2)}_0=0 , \\
			A^{(2)}_1= 1 , \\
			A^{(2)}_2= \nu^{(2)}_2 , \\
			\hspace{.895cm} \vdots \\
			A^{(2)}_{p-1}=\nu^{(2)}_{p-1} ,
		\end{cases}
		&& \cdots &&
		\begin{cases}
			A^{(p)}_0 =0 , \\
			\hspace{.915cm} \vdots \\
			A^{(p)}_{p-2} = 0 , \\
			A^{(p)}_{p-1} = 1,
		\end{cases}
	\end{aligned}
\end{align}
with \( \nu^{(i)}_{j} \) arbitrary constants, and 
\begin{align}
	\label{eq:initcondtypeII}
	\begin{aligned}
		\begin{cases}
			B^{(1)}_0=1 , \\
			B^{(1)}_1= \xi^{(1)}_1 , \\
			\hspace{.895cm} \vdots \\
			B^{(1)}_{q-1}=\xi^{(1)}_{q-1} ,
		\end{cases}
		&&
		\begin{cases}
			B^{(2)}_0=0 , \\
			B^{(2)}_1= 1 , \\
			B^{(2)}_2= \xi^{(2)}_2 , \\
			\hspace{.895cm} \vdots \\
			B^{(2)}_{q-1}=\xi^{(2)}_{q-1} ,
		\end{cases}
		&& \cdots &&
		\begin{cases}
			B^{(q)}_0 =0 , \\
			\hspace{.915cm} \vdots \\
			B^{(q)}_{q-2} = 0 , \\
			B^{(q)}_{q-1} = 1,
		\end{cases}
	\end{aligned}
\end{align}
with \( \xi^{(i)}_{j} \) also arbitrary. We also define the initial condition matrices
\begin{align}
	\label{eq:ic}
	\begin{aligned}
		\nu&\coloneq \begin{bNiceMatrix}
			1& 0 & \Cdots& && 0 \\
			\nu^{(1)}_1 & 1 & \Ddots&& & \Vdots \\
			\Vdots & \Ddots[shorten-start=-3pt,shorten-end=-3pt] & \Ddots& && \\
			&&&& &\\&&&&&0\\
			\nu^{(1)}_{p-1} &\Cdots& && \nu^{(p-1)}_{p-1}& 1 
		\end{bNiceMatrix} ,&
		\xi&\coloneq \begin{bNiceMatrix}
			1& 0 & \Cdots& && 0 \\
			\xi^{(1)}_1 & 1 & \Ddots&& & \Vdots \\
			\Vdots & \Ddots[shorten-start=-3pt,shorten-end=-3pt] & \Ddots& && \\
			&&&& &\\&&&&&0\\
			\xi^{(1)}_{q-1} &\Cdots& && \xi^{(q-1)}_{q-1}& 1 
		\end{bNiceMatrix}.
	\end{aligned}
\end{align}

The recursion polynomials are uniquely determined by the initial conditions \eqref{eq:initcondtypeI} and \eqref{eq:initcondtypeII} and by the recursion relations
\begin{align}\label{eq:recursion_dual_A}
	A^{(a)}_{n-q} T_{n-q,n}+ \cdots +A^{(a)}_{n+p} T_{n+p,n}&= x A^{(a)}_{n}, & n &\in\{0,1,\ldots\}, & a &\in\{1,\dots, p\}, & A_{-q}^{(a)} &=\dots=A^{(a)}_{-1}= 0,\\
	\label{eq:recursion_B}
	T_{n,n-p}B^{(b)}_{n-p} + \cdots + T_{n,n+q}B^{(b)}_{n+q} &= x B^{(b)}_{n}, & n &\in\{0,1,\ldots\}, & b&\in\{1,\dots, q\}, & B_{-p}^{(b)} &=\dots=B^{(b)}_{-1}= 0.
\end{align}

For the semi-infinite matrix \(T\), we consider the polynomials \(P_N(x)\) as the characteristic polynomials of the truncated matrices \(T^{[N-1]}\), i.e.,
\begin{align*}
	P_{N}(x)&\coloneq\begin{cases}
		1, & N=0,\\\det\big(xI_N-T^{[N-1]}\big), & N\in\N.
	\end{cases}
\end{align*}

Let us introduce the following matrices of left and right recursion polynomials:
\[ \begin{aligned}
	A_N &\coloneq	\begin{bNiceMatrix}
		A^{(1)}_N& \Cdots & A^{(1)}_{N+p-1}	 \\[2pt]
		\Vdots & & \Vdots \\[2pt]
		A^{(p)}_N	& \Cdots & A^{(p)}_{N+p-1}
	\end{bNiceMatrix}, & 
	B_N &\coloneq	\begin{bNiceMatrix}
		B^{(1)}_N & \Cdots & B^{(q)}_N \\[2pt]
		\Vdots & & \Vdots \\[2pt]
		B^{(1)}_{N+q-1} & \Cdots & B^{(q)}_{N+q-1}
	\end{bNiceMatrix},& N&\in \N_0 ,
\end{aligned}\]
and the following products
\[ \begin{aligned}
	\alpha_N &\coloneq (-1)^{(p-1)N}T_{p,0}\cdots T_{N+p-1,N-1}, &
	\beta_N &\coloneq (-1)^{(q-1)N}T_{0,q}\cdots T_{N-1,N+q-1} , & N&\in \N,
\end{aligned}\]
with \( \alpha_0=\beta_0=1 \).
Recall that since the entries on the extreme diagonals do not vanish, we have
\( \alpha_N,\beta_N\neq 0 \).

For \( N\in\N_0 \), the characteristic polynomials and the determinants of the left and right recursion polynomial blocks satisfy
\begin{align*}
	P_N(x)&=\alpha_N\det A_N(x)=\beta_N\det B_N(x).
\end{align*}

We now consider determinantal polynomials constructed in terms of determinants of left and right recursion polynomials that yield left and right eigenvectors of $T^{[N]}$.
Let us introduce the determinantal polynomials
\begin{align}\label{eq:QNn}
	\begin{aligned}
		Q_{n,N}&\coloneq\begin{vNiceMatrix}
			A^{(1)}_{n} & \Cdots[shorten-start=-3pt] & A^{(p)}_{n} \\[2pt]
			A^{(1)}_{N+1} & \Cdots[shorten-start=-1pt]  & A^{(p)}_{N+1} \\[2pt]
			\Vdots & & \Vdots \\[2pt]
			A^{(1)}_{N+p-1} & \Cdots[shorten-start=-3pt]  & A^{(p)}_{N+p-1}
		\end{vNiceMatrix},&
		R_{n,N}&\coloneq\begin{vNiceMatrix}
			B^{(1)}_{n} & \Cdots[shorten-start=-3pt]  & B^{(q)}_{n} \\[2pt]
			B^{(1)}_{N+1} & \Cdots[shorten-start=-1pt]  & B^{(q)}_{N+1} \\[2pt]
			\Vdots & & \Vdots \\[2pt]
			B^{(1)}_{N+q-1} & \Cdots[shorten-start=-3pt]  & B^{(q)}_{N+q-1}
		\end{vNiceMatrix},
	\end{aligned}
\end{align}
the semi-infinite row and column vectors
\[	\begin{aligned}
	Q_N&\coloneq\left[\begin{NiceMatrix}
		Q_{0,N} &Q_{1,N} &\Cdots
	\end{NiceMatrix}\right], &
	R_N&\coloneq\left[\begin{NiceMatrix}
		R_{0,N} \\R_{1,N} \\\Vdots
	\end{NiceMatrix}\right],
\end{aligned}\]
and corresponding truncations
\[ 	\begin{aligned}
	Q^{\langle N\rangle}&\coloneq \begin{bNiceMatrix}
		Q_{0,N} &Q_{1,N}&\Cdots & Q_{N,N}
	\end{bNiceMatrix}, & 
	R^{\langle N\rangle}&\coloneq \begin{bNiceMatrix}
		R_{0,N} \\R_{1,N}\\\Vdots \\ R_{N,N}
	\end{bNiceMatrix}.
\end{aligned}\]

There is a generalized Christoffel--Darboux formula for the determinantal polynomials and the characteristic polynomial of a banded matrix:
\begin{enumerate}
	\item The biorthogonal families of left and right eigenvectors $\big\{w^{\langle N\rangle}_k\big\}_{k=0}^{N}$ and $\big\{u^{\langle N\rangle}_k\big\}_{k=0}^{N}$ are
	\begin{equation}\label{eq:eigenvectors}
		\begin{aligned}
			w^{\langle N\rangle}_{k}&=\frac{ Q^{\langle N\rangle}\big(\lambda^{[N]}_k\big)}{\beta_N\sum_{l=0}^{N}Q_{l,N}\big(\lambda^{[N]}_k\big)R_{l,N}\big(\lambda^{[N]}_k\big)}, &
			u^{\langle N\rangle}_{k}&=\beta_N R^{\langle N\rangle}\big(\lambda^{[N]}_k\big).
		\end{aligned}
	\end{equation}
	that is, $w^{\langle N\rangle}_{k}u^{\langle N\rangle}_{l}=\delta_{k,l}$.
	
	\item In terms of the characteristic polynomial, the left eigenvector entries can be written as
	\[	\begin{aligned}
		w^{\langle N\rangle}_{k,n}&=
		\frac{ \alpha_N Q_{n,N}\big(\lambda^{[N]}_k\big)
		}{
			P_{N}\big(\lambda^{[N]}_k\big)P'_{N+1}\big(\lambda^{[N]}_k\big)}.
	\end{aligned}\]
	
	\item The corresponding matrices $\mathscr U$ (with columns the right eigenvectors $u_k$ arranged in the standard order) and 
	$\mathscr W$ (with rows the left eigenvectors $w_k$ arranged in the standard order) satisfy
	$ \mathscr U\mathscr W=\mathscr W\mathscr U=I_{N+1}$.
	
	\item In terms of the diagonal eigenvalue matrix $D=\diag\big(\lambda^{[N]}_0,\lambda^{[N]}_1,\dots,\lambda^{[N]}_{N}\big)$ we have
	\begin{align}\label{eq:UDnW=Jn}
		\begin{aligned}
			\mathscr UD^n\mathscr W&=\big(T^{[N]}\big)^n, & n&\in\N_0.
		\end{aligned}
	\end{align}
\end{enumerate}

The so-called Christoffel numbers $\mu^{[N]}_{k,a}$ and $\rho^{[N]}_{k,b}$, $a\in\{1,\dots,p\}$ and $b\in\{1,\dots,q\}$, see \cite{BFM1}, allow us to write
\[	\begin{aligned}
	w^{\langle N\rangle}_{k,n}&= A_{n-1}^{(1)}\big(\lambda^{[N]}_k\big)\mu^{[N]}_{k,1} +\cdots +A_{n-1}^{(p)}\big(\lambda^{[N]}_k\big)\mu^{[N]}_{k,p},\\
	\label{eq:discrete_linear_form_typeII}
	u^{\langle N\rangle}_{k,n}&= B_{n-1}^{(1)}\big(\lambda^{[N]}_k\big)\rho^{[N]}_{k,1} +\cdots +B_{n-1}^{(q)}\big(\lambda^{[N]}_k\big)\rho^{[N]}_{k,q}.
\end{aligned}\]

The Christoffel numbers are connected with the left and right eigenvector entries and the initial conditions as follows:
\begin{align*}
	\begin{aligned}
		\begin{bNiceMatrix}
			\mu^{[N]}_{k,1} \\[5pt]
			\mu^{[N]}_{k,2} \\
			\Vdots
			\\
			\mu^{[N]}_{k,p}
		\end{bNiceMatrix}
		&= \nu^{-1} \begin{bNiceMatrix}
			w^{\langle N\rangle}_{k,1} \\[5pt]
			w^{\langle N\rangle}_{k,2} \\
			\Vdots \\
			w^{\langle N\rangle}_{k,p}
		\end{bNiceMatrix}, &
		\begin{bNiceMatrix}
			\rho^{[N]}_{k,1} \\[5pt]
			\rho^{[N]}_{k,2} \\
			\Vdots
			\\
			\rho^{[N]}_{k,q}
		\end{bNiceMatrix}
		&= \xi^{-1} \begin{bNiceMatrix}
			u^{\langle N\rangle}_{k,1} \\[5pt]
			u^{\langle N\rangle}_{k,2} \\
			\Vdots \\
			u^{\langle N\rangle}_{k,q}
		\end{bNiceMatrix}.
	\end{aligned}
\end{align*}

In \cite{BFM1} we proved that, for an appropriate choice of the initial conditions for the recursion polynomials in terms of the bidiagonal factorization, the Christoffel numbers are positive, i.e.
\[	\begin{aligned}
	\rho^{[N]}_{k,b}&>0, &\mu^{[N]}_{k,a}&>0, &k&\in\{0,1,\dots,N\}, &a&\in\{1,\dots, p\}, & b&\in\{1,\dots,q\}.
\end{aligned}\]

Let us consider the following step functions:
\begin{align*}
	\psi^{[N]}_{b,a}&\coloneq \begin{cases}
		0, & x<\lambda^{[N]}_{N},\\[2pt]
		\rho^{[N]}_{1,b}\mu^{[N]}_{1,a}+\cdots+\rho^{[N]}_{k,b}\mu^{[N]}_{k,a}, & \lambda^{[N]}_{k+1}\leqslant x< \lambda^{[N]}_{k}, \quad k\in\{0,1,\dots,N-1\},\\[2pt]
		\rho^{[N]}_{1,b}\mu^{[N]}_{1,a}+\cdots+\rho^{[N]}_{N+1,b}\mu^{[N]}_{N+1,a} ,
		& x \geqslant \lambda^{[N]}_{0}.
	\end{cases}
\end{align*}
The last step of these step functions is bounded. Specifically, for $a\in\{1, \dots, p\}$ and $b\in\{1, \dots, q\}$, we have 
\begin{align}\label{eq:bound}
	\rho^{[N]}_{0,b}\mu^{[N]}_{0,a}+\cdots+\rho^{[N]}_{N,b}\mu^{[N]}_{N,a} = (\xi^{-1}I_{q,p}\nu^{-\top})_{b,a}.
\end{align}
Here $I_{q,p}\in\R^{q\times p}$ is the rectangular identity matrix, with $(I_{q,p})_{k,l}=\delta_{k,l}$. 

Let us introduce the $q\times p$ matrix
$\Psi^{[N]}\coloneq\begin{bNiceMatrix}[small]
	\psi^{[N]}_{1,1}&\Cdots &\psi^{[N]}_{1,p}\\
	\Vdots & &\Vdots\\
	\psi^{[N]}_{q,1}&\Cdots &\psi^{[N]}_{q,p}
\end{bNiceMatrix}$
and the corresponding $q\times p$ matrix of discrete Lebesgue--Stieltjes measures supported at the zeros of $P_{N+1}$,
\begin{align}\label{eq:discrete_mesures}
	\d\Psi^{[N]}=\begin{bNiceMatrix}
		\d\psi ^{[N]}_{1,1}&\Cdots &	\d\psi^{[N]}_{1,p}\\
		\Vdots & &\Vdots\\
		\d\psi^{[N]}_{q,1}&\Cdots &	\d\psi^{[N]}_{q,p}
	\end{bNiceMatrix}=\sum_{k=1}^{N+1}\begin{bNiceMatrix}
		\rho^{[N]}_{k,1}\\\Vdots \\	\rho^{[N]}_{k,q}
	\end{bNiceMatrix}\begin{bNiceMatrix}
		\mu^{[N]}_{k,1} & \Cdots & \mu^{[N]}_{k,p}
	\end{bNiceMatrix}\delta\big(x-\lambda^{[N]}_k\big).
\end{align}

The following biorthogonality relations hold:
\begin{align*}
	\sum_{a=1}^p\sum_{b=1}^{q} \int B^{(b)}_n(x)\d\psi^{[N]}_{b,a}(x)A^{(a)}_{m}(x)&=\delta_{n,m}, &n,m&\in\{0,\dots,N\}.
\end{align*}
The following orthogonality relations, for $m \in\{1,\dots, N\}$, are satisfied:
\begin{align*}
	\sum_{a=1}^p\int x^n\d\psi ^{[N]}_{b,a}(x)A^{(a)}_{m}(x)&=0, & 
	n&\in\left\{0,\dots,\left	\lceil \frac{m+1-b}{q}\right \rceil -1\right\},&	b&\in\{1,\dots,q\},\\
	\sum_{b=1}^q\int B^{(b)}_{m}(x)\d\psi ^{[N]}_{b,a}(x)x^n&=0, 
	& n&\in\left\{0,\dots,\left	\lceil \frac{m+1-a}{p}\right \rceil -1\right\}, & a&\in\{1,\dots,p\}.
\end{align*}
Therefore, the recursion polynomials are discrete multiple orthogonal polynomials of mixed type on the step-line. Moreover, we have the spectral representation
\[
\begin{aligned}
	(T^k)_{m,n}&=	\sum_{a=1}^p\sum_{b=1}^{q} \int B^{(b)}_n(x)x^k\d\psi^{[N]}_{b,a}(x)A^{(a)}_{m}(x),
&& m,n\in\{0,\dots,N\}.
\end{aligned}
\]

Hence, since the Christoffel coefficients are positive and the step functions are uniformly bounded in $N$, there exists a matrix of measures $\d\Psi$ as the large-$N$ limit of the truncated discrete measures $\d\Psi^{[N]}$. This constitutes our Favard-type spectral theorem:
\[
(T^k)_{m,n}=	\sum_{a=1}^p\sum_{b=1}^{q} \int B^{(b)}_n(x)x^k\d\psi_{b,a}(x)A^{(a)}_{m}(x),
\qquad m,n\in\N_0,
\]
together with the following multiple orthogonality relations of mixed type: for $m\in\N$,
\[\begin{aligned}
	\sum_{a=1}^p\int x^n\d\psi_{b,a}(x)A^{(a)}_{m}(x)&=0, & 
	n&\in\left\{0,\dots,\left	\lceil \frac{m+1-b}{q}\right \rceil -1\right\},&	b&\in\{1,\dots,q\},\\
	\sum_{b=1}^q\int B^{(b)}_{m}(x)\d\psi_{b,a}(x)x^n&=0, 
	& n&\in\left\{0,\dots,\left	\lceil \frac{m+1-a}{p}\right \rceil -1\right\}, & a&\in\{1,\dots,p\}.
\end{aligned}\]

\section*{Conclusions and outlook}

In this paper we have developed a spectral approach to discrete-time Markov chains whose transition matrices are bounded, banded and admit a stochastic bidiagonal factorization. This framework extends the classical birth--and--death setting to genuinely banded dynamics with several upward and downward jumps, while preserving a strong positivity structure in the sense of total positivity and oscillation theory.

A central result is the construction of a canonical stochastic bidiagonal factorization associated with any positive bidiagonal factorization of a banded stochastic matrix. This normalization, implemented through explicit diagonal conjugations, provides a transparent probabilistic interpretation of the dynamics as a concatenation of elementary birth and death mechanisms, and establishes a direct bridge between the oscillatory structure of the underlying operator and the Markov evolution.

Using the spectral Favard theorem for banded matrices with positive bidiagonal factorization, we have derived Karlin--McGregor type spectral representations in terms of  multiple orthogonal polynomials of mixed-type and matrix-valued spectral measures. In the finite case, this yields explicit formulas for transition probabilities, first-passage generating functions and stationary distributions, together with geometric convergence rates governed by the second largest eigenvalue. In the countably infinite case, we obtain analogous integral representations and characterize recurrence, transience and ergodicity through precise spectral conditions, including a mass-at-one criterion for positive recurrence.

From a broader perspective, this work shows that mixed-type multiple orthogonal polynomials provide a natural and flexible spectral language for the analysis of banded Markov chains beyond the tridiagonal regime. The results unify and clarify several previously studied Hessenberg and tetradiagonal models, and correct or refine certain aspects of earlier treatments by embedding them into a more general and robust banded framework.

Several directions for future research appear naturally. On the probabilistic side, it would be of interest to explore explicit models and applications, including random walks with internal degrees of freedom or multi-urn constructions associated with the stochastic bidiagonal factors. On the analytical side, extensions to unbounded banded operators, continuous-time Markov processes, or perturbations breaking exact bidiagonal factorizability merit further investigation. Finally, the interplay between total positivity, integrable structures and stochastic dynamics suggested by the present approach points toward deeper connections with non-commutative probability, integrable Markov processes and scaling limits, which we plan to address in future work.

\section*{Acknowledgments}
AB acknowledges Centro de Matemática da Universidade de Coimbra UID/MAT/00324/2020, funded by the Portuguese Government through FCT/MEC and co-funded by the European Regional Development Fund through the Partnership Agreement PT2020.

AF acknowledges CIDMA Center for Research and Development in Mathematics and Applications (University of Aveiro) and the Portuguese Foundation for Science and Technology (FCT) within project UIDB/MAT/UID/04106/2020 and UIDP/MAT/04106/2020.

MM acknowledges Spanish ``Agencia Estatal de Investigación'' research projects [PID2021- 122154NB-I00], \emph{Ortogonalidad y Aproximación con Aplicaciones en Machine Learning y Teoría de la Probabilidad} and  [PID2024-155133NB-I00],  \emph{Ortogonalidad, aproximación e integrabilidad: aplicaciones en procesos estocásticos clásicos y cuánticos}.

\printbibliography

\begin{thebibliography}{99}
	
	\bibitem{afm}
	Carlos Álvarez-Fernández, Ulises Fidalgo, and Manuel Mañas,
	\emph{Multiple orthogonal polynomials of mixed type: Gauss--Borel factorization and the multi-component 2D Toda hierarchy},
	Advances in Mathematics~\textbf{227} (2011) 1451–1525.
	
	
	\bibitem{JP} Amilcar Branquinho, Juan EF Díaz, Ana Foulquié-Moreno, Manuel Mañas, and Carlos Álvarez-Fernández, \emph{Jacobi-Piñeiro random walks}. Revista de la Real Academia de Ciencias Exactas, Físicas y Naturales. Serie A. Matemáticas \textbf{118} (2024), article number 15. 
	
	\bibitem{BFM1} Amílcar Branquinho, Ana Foulquié-Moreno, and Manuel Mañas, \emph{Spectral theory for bounded banded matrices with positive bidiagonal factorization and mixed multiple orthogonal polynomials}, Advances in Mathematics \textbf{434} (2023) 109313.
	
	\bibitem{BFM2} Amílcar Branquinho, Ana Foulquié-Moreno, and Manuel Mañas, \emph{Positive bidiagonal factorization of tetradiagonal Hessenberg matrices},  Linear Algebra and its Applications \textbf{667} (2023) 132-160.
	
	\bibitem{BFM3} Amílcar Branquinho, Ana Foulquié-Moreno, and Manuel Mañas,  \emph{Oscillatory banded Hessenberg matrices, multiple orthogonal polynomials and random walks},  Physica Scripta \textbf{98} (2023).
	
\bibitem{BFM4}	Amílcar Branquinho, Ana Foulquié-Moreno, and Manuel Mañas,
	\emph{Banded totally positive matrices and normality for mixed multiple orthogonal polynomials},
	preprint, 2024.
	Available at \url{https://arxiv.org/abs/2404.13965}.
	
	\bibitem{bremaud} Pierre Brémaud, \emph{Markov Chains} (Second Edition), Texts in Applied Mathematics \textbf{31}, Springer, 2020.
	
	
	\bibitem{Fallat} Shaun M. Fallat and Charles R. Johnson, \emph{Totally Nonnegative Matrices}, Princeton University Press, Princeton, 2011.


\bibitem{feller}
William Feller,
\emph{An Introduction to Probability Theory and Its Applications, Vol.~1},
3rd edition, John Wiley \& Sons, New York, 1968.

\bibitem{gallager} Robert G. Gallager, \emph{Stochastic Processes. Theory and Applications}.Cambridge University Press, 2013.
	
	\bibitem{Gantmacher} Felix Gantmacher and Mark G. Krein, \emph{Oscillation Matrices and Kernels and Small Vibrations of Mechanical Systems}, Revised Edition, AMS Chelsea Publishing, Providence, 2002. 
	
	\bibitem{haggstrom} Olle Häggström, \emph{Finite Markov Chaimsand Algorithmic Applications}, London Mathematical Societ Stident Texts \textbf{52}, Cambridge University Press, 2002.
	
	
	
	\bibitem{Ismail} 
	Mourad E. H. Ismail, 
	\emph{Classical and Quantum Orthogonal Polynomials in One Variable}, 
	Encyclopedia of Mathematics and its Applications \textbf{98}, Cambridge University Press, 2009.
	
	\bibitem{Karlin} Samuel  Karlin, \emph{Total Positivity,} Stanford University Press, Stanford, 1968.
	
	\bibitem{KarlinMcGregor}
	Samuel Karlin and James McGregor,
	\emph{Random walks},
	Illinois Journal of Mathematics \textbf{3} (1959), 66--81.
	
	\bibitem{kirwood} James E. Kirwood, \emph{Markov Processes}, CRC Press, 2024.
	
	\bibitem{andrei_walter}
	Andrei Martínez-Finkelshtein and Walter Van Assche,
	\emph{What is a multiple orthogonal polynomial?},
	Notices of the AMS \textbf{63}:9 (2016)
	1029–1031. 

	
	\bibitem{nikishin_sorokin}
	Evgenii M. Nikishin and Vladimir N. Sorokin,
	\emph{Rational Approximations and Orthogonality},
	Translations of Mathematical Monographs \textbf{92},
	American Mathematical Society, Providence, 1991.
	
	
	\bibitem{Pinkus book} Allan Pinkus, \emph{Totally Positive Matrices}, Cambridge Tracts in Mathematics \textbf{181}, Cambridge University Press, Cambridge, 2010.
	
	
\end{thebibliography}

\end{document}
	
\end{document}